\documentclass[11pt]{amsart}
\usepackage{amsthm,amssymb,url,hyperref}
\usepackage{fullpage}
\usepackage{stmaryrd}
\usepackage[all]{xy}
\usepackage{color}

\theoremstyle{plain}
\newtheorem{theorem}{Theorem}[section]

\newtheorem{corollary}[theorem]{Corollary}

\newtheorem{lemma}[theorem]{Lemma}

\newtheorem{proposition}[theorem]{Proposition}

\theoremstyle{definition}
\newtheorem{definition}[theorem]{Definition}
\newtheorem{problem}[theorem]{Problem}
\newtheorem{example}[theorem]{Example}
\newtheorem{remark}[theorem]{Remark}

\def\ker#1{\mathrm{ker}(#1)}

\def\Fix#1{\mathrm{Fix}(#1)}
\def\Soc#1{\mathrm{Soc}(#1)}
\def\aut#1{\mathrm{Aut}(#1)}

\def\Pol{\mathrm{Pol}}

\def\comm#1#2{{\llbracket #1, #2 \rrbracket }}

\def\setof#1#2{\{#1\, : \,#2\}}

\def\aut{\mathrm{Aut}}

\def\cg#1{\equiv_\alpha}
\newcommand*\xbar[1]{%
   \hbox{%
     \vbox{%
       \hrule height 0.5pt 
       \kern0.5ex
       \hbox{%
         \kern-0.1em
         \ensuremath{#1}%
         \kern-0.1em
       }%
     }%
   }%
} 

\title{Central nilpotency of skew braces}

\author{Marco Bonatto}

\address[Bonatto]{Department of Mathematics and Computer Science, University of Ferrara, Via Macchiavelli 30, 44121 Ferrara, Italy}

\email{marco.bonatto.87@gmail.com}

\author{P\v remysl Jedli\v cka}
\address[Jedli\v cka]{Department of Mathematics, Faculty of Engineering, Czech University of Life Sciences, Kam\'yck\'a 129, 16521 Praha 6, Czech Republic}
\email{jedlickap@tf.czu.cz}

\begin{document}



%

\maketitle

%
%
%

\section*{Abstract}
Skew braces are algebraic structures related to the solutions of the set-theoretic quantum Yang-Baxter equation. We develop the central nilpotency theory for such algebraic structures in the sense of Freese-McKenzie \cite{comm} and we compare the universal algebraic notion of central nilpotency with the notion of right and left $*$-nilpotency developed in \cite{NilpotentType}.


\section*{Introduction}

A discrete version of the braid equation was introduced in \cite{drinfeld} as the set-theoretical Yang--Baxter equation (YBE). A pair $(X,r)$ where $X$ is a set and $r:X\times X\to X\times X$ is a map, is a \emph{set-theoretical solution to the Yang--Baxter equation} if
\begin{equation}\label{YBE}
(id_X \times r)(r\times id_X)(id_X \times r)=(r\times id_X)(id_X \times r)(r\times id_X)\tag{YBE}
\end{equation}
holds.
A solution $(X,r)$ is {\it non-degenerate} if the map $r$ is defined as
\begin{equation}\label{non-deg}
r:    X\times X\longrightarrow X\times X,\quad (x,y)\mapsto (\sigma_x(y),\tau_y(x)),
\end{equation}
where $\sigma_x,\tau_x$ are permutations of $X$ for every $x\in X$. The family of non-degenerate solutions have been studied by several authors \cite{EGS,MR1637256, MR1769723,MR1809284}. 

Solutions of the \eqref{YBE} can be encoded using certain algebraic structures. Concerning \emph{non-degenerate involutive solutions} (i.e. solutions satisfying $r^2=id_{X\times X}$), Rump introduced the ring-like binary algebraic structures called {\it (left) braces} in \cite{Rump1}. Non-involutive solution are captured by \emph{skew (left) braces} and they were defined by Guarnieri and Vendramin in \cite{Guarneri}.

Finding non-degenerate solutions to \eqref{YBE} is a task closely related to the classification problem for skew braces. Indeed, on one hand to any non-degenerate solution is associated a permutation group with a canonical structure of skew brace. On the other hand all the non-degenerate solutions with a given associated skew brace can be described  \cite{MR3835326}. Thus, the classification of skew braces is the first step to the study of non-degenerate solutions to \eqref{YBE}.

Some group-like definitions of nilpotency for left braces appeared already in~ \cite{Rump1}. This definitions showed up to be practical since they are tied with some algebraic properties of the underlying solution of the Yang-Baxter equation~\cite{Engel}.
These notions have been naturally generalized for
skew braces in \cite{NilpotentType} by using a term operation~$*$ that plays the role of the commutator operation for groups.

In universal algebra, we have a well-established notion of nilpotency and commutators~\cite{comm} and it is thus interesting to study the specialization of
these notions to the theory of skew braces. The notion of the center and of the (central) nilpotency translate nicely and they lead to yet another definition of
nilpotency for skew braces. On the other hand, the operation~$*$ itself is not
sufficient to define the commutator in skew braces.

The article is organized as follows: in Section~1 we define skew braces and we recall some basic properties as well as the notions of nilpotency defined elsewhere.
In Section~2 we define the center of a skew brace and the lower and the upper central series tied to the center. We then study properties of these series and
of the central nilpotency defined through the series.
In Section~3 we recall the abstract definitions of the center and of the central nilpotency and we prove that they agree with the notions defined in Section~2.

\section{Basic facts about skew braces}

\subsection*{Skew braces}
A {\it skew brace} is a set with two binary operations $(A,+,\circ)$ where both $(A,+)$ and $(A,\circ)$ are groups and 
\begin{equation}\label{axiom}
x\circ (y+ z)=(x\circ y)- x+ (x\circ z)
\end{equation}
holds for every $x,y,z\in A$. The mapping defined as
\begin{displaymath}
\lambda: (A,\circ)\longrightarrow \aut(A,+), \quad a\mapsto \lambda_a
\end{displaymath}
where $\lambda_x(y)=-x+ (x\circ y)$, is a group homomorphism between $(A,\circ)$ and $\aut(A,+)$. Therefore equation \eqref{axiom} can be written as 
\begin{equation}\label{axiom2}
x\circ (y+z)=(x\circ y)+\lambda_x(z).
\end{equation}

It can be proved that both groups share the same neutral element, we shall denote it by~$0$. We shall denote by $-x$ the inverse element of $a$ with respect to $+$ and by $\bar x$ the inverse the inverse element of $x$ with respect to $\circ$. We say that the brace $(A,+,\circ)$ is of {\it $\phi$-type} if the group $(A,+)$ has the property $\phi$ (e.g. we say that $(A,+,\circ)$ is of nilpotent-type if $(A,+)$ is nilpotent). Braces, as defined by Rump, are exactly skew braces of abelian type (we stick to such terminology in the rest of the paper).

Skew braces were introduced in \cite{Guarneri} as a tool for constructing set-theoretic solutions of Yang-Baxter equation (Rump introduced skew braces of Abelian type in \cite{Rump1}).

\begin{lemma}\label{condition on lambda} \cite{SmoktVend}
	Let $(A,+)$ be a group, $\lambda:A\longrightarrow \aut(A,+)$ and let $x\circ y=x+\lambda_x(y)$ for $x,y\in A$. The following are equivalent:
\begin{itemize}
\item[(i)] $(A,+, \circ)$ is a skew brace.
\item[(ii)] $\lambda_{x\circ y}=\lambda_{x+\lambda_x(y)}=\lambda_x \lambda_y$ for every $x,y\in A$.
\end{itemize}	
	
\end{lemma}

A {\it left ideal} of a skew brace $(A,+,\circ)$ is a subgroup $H$ of $(A,+)$ such that $\lambda_a(H)\leq H$ for every $a\in H$ \cite{Guarneri}. Note that every characteristic subgroup of $(A,+)$ is a left ideal and every left ideal is a subgroup of $(A,\circ)$. The subgroup 
$$\Fix{A}=\setof{x\in A}{\lambda_y(x)=x, \, \text{ for all }y\in A}$$ 
is clearly a left ideal.

A quotient of a skew brace $(A,+,\circ)$ is a quotient of both the groups $(A,+)$ and $(A,\circ)$. Therefore, any quotient is given by a subgroup $I$ which is normal with respect to both group structures and that satisfies $a+I=a\circ I$ for every $a\in A$. This last condition is controlled by the $\lambda$ mappings. Indeed:
\begin{align}\label{prop of ideals}
x+I=x\circ I \, &\Leftrightarrow \, -x+ (x\circ i)=\lambda_x(y)\in I\, \text{ for every }\, y\in I.
\end{align}
 We call a subgroup satisfying \eqref{prop of ideals} an {\it ideal} of $A$. Hence, an ideal is a normal subgroup $I$ of both $(A,+)$ and $(A,\circ)$ such that $\lambda_x(I)\leq I$ for every $x\in I$ (equivalently $I$ is a left ideal which is also a normal subgroup of $(A,+)$ and $(A,\circ)$). For instance, the socle of $A$ defined as $\Soc{A}=\ker{\lambda}\cap Z(A,+)$ is an ideal.

\subsection*{The $*$ operation}

For a skew brace $A$, following \cite{Rump1, NilpotentType},
 we define the following binary operation \begin{equation}
x*y =-x+ (x\circ y)-y= \lambda_x(y)-y
\end{equation}
for every $x,y\in A$. This operation measures the difference between the two operations $+$ and $\circ$ (through the $\lambda$ mappings). 
%
	Indeed, note that $x*y=0$ if and only if $\lambda_x(y)=y$ if and only if $x+y=x\circ y$.
The elements in the subset
\begin{equation}\label{pseudo socle} 
\Fix{A}\cap \ker{\lambda}=\setof{x\in A}{x*y=y*x=0, \text{ for all } y\in A}
\end{equation}
are the elements such that $x\circ y=x+y$ and $y\circ x=y+x$ for every $y\in A$. Clearly it is a subgroup of $(A,+)$ and of $(A,\circ)$ and it is trivially invariant under $\lambda_x$ for every $x\in A$, so it is a left ideal. 
Moreover, \begin{eqnarray}\label{comm1}
(a+x)*(b+y)&=& \lambda_{a+x}(b+y)-y- b= \lambda_{a\circ x}(b)+\lambda_{a\circ x}(y)-y- b=\notag\\
&=&\lambda_a(b)-b=a*b
\end{eqnarray}
for every $x,y\in \Fix{A}\cap \ker{\lambda}$ and $a,b\in A$. Note that $\Fix{A}\cap \ker{\lambda}$ is a left ideal, and it is also an ideal for all the skew braces in the database of the GAP package \cite{GAP}. We do not know if this is true in general. 


The $*$ operation satisfies the following identity
\begin{equation}\label{eq:star-distr}
x*(y+z)=\lambda_x(y+z)-z- y=\lambda_x(y)-y+ y+ \lambda_x(z) -z-y =x*y+y+x*z-y
\end{equation}
which is analog to the equation satisfied by the commutator in groups.

\begin{remark}\label{star for op braces}
Consider a skew-brace $(B,+,+^{op})$. Then $x*y=[x,y^{-1}]_+$. This example shows that the $*$ operation plays a role similar to the element-wise commutator in groups.
\end{remark}

Let $I,J$ be subsets of $A$. We define 
$$I*J=\langle i*j, \, i\in I,\, j\in J\rangle_+.$$ If $I$ and $J$ are ideals then, by definition, 
$I*J\subseteq I\cap J$. It is well known, that a product of two subgroups is a subgroup if and only if the subgroups permute. It is natural to expect a similar behavior for the $*$-product of two ideals.

\begin{lemma}\label{I*J+J*I is sub} 
Let $(A,+,\circ)$ be a skew brace and $I,J$ be ideals. Then $I*J+J*I$ is a subgroup of~$(A,+)$.
\end{lemma}

\begin{proof} We need to prove $J*I+I*J\subseteq I*J+J*I$.
According to \eqref{eq:star-distr}, we have
\begin{multline*}j_1*i_1+i_2*j_2=
-(i_2*(j_1*i_1))+i_2*(j_1*i_1)+j_1*i_1+i_2*j_2 -j_1*i_1+j_1*i_1=\\
-(i_2*(j_1*i_1))+i_2*(j_1*i_1+j_2)+j_1*i_1\in
I*J+I*J+J*I.
\end{multline*}
We need also $-(I*J+J*I)\subseteq I*J+J*I$. But clearly
\[-(i_1*j_2+j_2*i_1)=-(j_2*i_1)-(i_1*j_2)\in J*I+I*J\subseteq I*J+J*I.\]
Then easily, by an induction, we get $\langle I*J, J*I\rangle_+=I*J+J*I$.
\end{proof}

	\begin{proposition}\label{left ideals from I and J}
		Let $(A,+,\circ)$ be a skew brace and $I,J$ be ideals. Then $I*J$, $[I,J]_+$ and $I*J+J*I$ are left ideals of $A$.
	\end{proposition}
	\begin{proof}
		Let $i\in I$, $j\in J$, $x\in A$. Then 
		\begin{eqnarray*}
			\lambda_x(i*j)&=&\lambda_x(\lambda_i(j)-j)=\lambda_x \lambda_i(j)-\lambda_x(j)=\\
			&=&\lambda_{x\circ i \circ \bar{x}}\lambda_x(j)-\lambda_x(j)=\\
			&=&(x\circ i \circ \bar{x})*\lambda_x(j)\in I*J.
			%
		\end{eqnarray*}
		Then $I*J$ is invariant under $\lambda_x$ for every $x\in A$ and then it is a left ideal. According to Lemma \ref{I*J+J*I is sub} $I*J+J*I$ is a subgroup. The sum of left ideals is again a left ideal, then $I*J+J*I$ is a left ideal as well.\\
		Clearly $[I,J]_+$ is a subgroup of $(A,+)$ and $\lambda_x([i,j]_+)=[\lambda_c(i),\lambda_c(j)]_+$ for every $x\in A$. Since $I$ and $J$ are left ideals, then $\lambda_x(i)\in I$ and $\lambda_x(j)\in J$. Therefore $[I,J]_+$ is a left ideal.
	\end{proof}

The $*$ operation provides a characterization of (left) ideals of a skew braces.

\begin{proposition}\label{ideals by bracket}\cite[Lemmas 1.8, 1.9]{NilpotentType}
	Let $(A,+, \circ)$ be a skew brace. A normal subgroup $I$ of $(A,+)$ is an ideal if and only if $I*A+A*I\leq I$.
\end{proposition}

%

In \cite{NilpotentType} two series of left ideals have been defined, for a skew brace $A$,  as
\begin{eqnarray}
A^1&=&A,\quad A^{n+1}= A^{n}*A,\\
A^{(1)}&=&A,\quad A^{(n+1)}=A^{(n)}*A
\end{eqnarray} 
A brace is said to be {\it right $*$-nilpotent} (resp. {\it left $*$-nilpotent}) if $A^{(n)}=0$ (resp. $A^{n}=0$) for some $n\in \mathbb{N}$. In \cite{Engel} the following series of left ideals for skew braces was defined:
$$A^{[1]}=A,\quad A^{[n+1]} =\left\langle \bigcup_{i=1}^n A^{[i]}*A^{[n+1-i]}\right\rangle_+.$$
It turned out that a skew brace $A$ is left and right $*$-nilpotent if and only if $A^{[n]}=0$ for some $n\in \mathbb{N}$. In such case we say that $A$ is $*$-nilpotent.

\section{Centrally Nilpotent Skew Braces}

We finished the last section with definitions
of some nilpotency series. In this section
we present yet another definition that mimics the behavior of the nilpotency of groups
and we compare this definition with the definitions above.

Let $(A,+,\circ)$ be a skew brace we define the {\it center} of $(A,+,\circ)$ as $\zeta(A)=\Soc{A}\cap \Fix{A}$. The center of $(A,+,\circ)$ is thus the set 
\begin{multline}\label{center2}
\zeta(A)=\setof{x\in A}{x*y=y*x=[x,y]_+=0,\, \text{for every } y\in A}\\
=\setof{x\in A}{x+y=y+x=x\circ y=y\circ x,\, \text{for every } y\in A}.
\end{multline}

Note that every subgroup of $(A,+)$ contained in $\zeta(A)$ is an ideal and that $\zeta(A)\leq Z(A,+)\cap Z(A,\circ)$.

\begin{remark} 
There are examples of skew braces $(A,+,\circ)$ with non-trivial socle and such that $Fix(A)=0$ and vice versa (such examples can be traced down in the database of the GAP package \cite{GAP}).
\end{remark}
\begin{definition}\label{nilpotent}
A skew brace $(A,+,\circ)$ is said to be {\it centrally nilpotent of class~$n$} if there exists a chain of ideals
\begin{equation}
0=I_0\leq I_1\leq \ldots\leq I_n=A,
\end{equation}
such that $I_{j+1}/I_j\leq \zeta(A/I_j)$ for every $0\leq j\leq n-1$.
\end{definition}
There are canonical chains of subgroups that measure the class of central nilpotency:

\begin{definition}
Let $A$ be a skew brace, $n\in \mathbb{N}$ and $I\subseteq A$. We define the {\em upper central series} of~$A$
and the {\em lower central series} of~$I$ as follows:
\begin{eqnarray}
\zeta_0(A)=0, &\quad& \zeta_n(A)=\setof{x\in A}{x*y,\, y*x\, ,[x,y]\in \zeta_{n-1}(A) \text{ for every } y\in A}\label{centers}\\
\Gamma_0(I)=I,&\quad &\Gamma_n(I)=\langle\Gamma_{n-1}(I)*A,\,A*\Gamma_{n-1}(I),\,[\Gamma_{n-1}(I),A]_+\rangle_+.\label{gammas}
\end{eqnarray}
In particular $\zeta_1(A)=\zeta(A)$ as defined in \eqref{center2}.
\end{definition}

\begin{example}
 Let $(A,+,+^{op})$ be a skew brace. Than
 the upper and lower central series of the
 skew brace are the same as the upper and lower central series of the group $(A,+)$
 since $*$ is the ordinary group commutator here,
 according to Remark~\ref{star for op braces}.
\end{example}

\begin{lemma}\label{lem:lower_central_series}
Let $(A,+,\circ )$ be a skew brace. Then $\zeta_n(A)$ is an ideal 
and $\zeta_{n+1}(A)/\zeta_n(A)=\zeta(A/\zeta_n(A))$ for every $n\in \mathbb{N}$.
\end{lemma}
\begin{proof}
If $n=1$ then $A*\zeta_1(A)+\zeta_1(A)*A=0$ and then $\zeta_1(A)$ is an ideal,
according to Proposition~\ref{ideals by bracket}. Assume by induction that $\zeta_n(A)$ is an ideal. The set $\zeta_{n+1}(A)$ is the preimage of $\zeta(A/\zeta_n(A))$ under the canonical projection onto $A/\zeta_n(A)$ then it is an ideal. 
\end{proof}

The upper central series is thus a series of ideals. On the other hand, we don't know
whether $\setof{\Gamma_n(A)}{n\in\mathbb{N}}$ are ideals or even left ideals.

\begin{lemma}\label{gamma and zeta}
	Let $A$ be a skew brace, $I\subseteq A$ and $n,k\in\mathbb{N}$. Then $\Gamma_n(I)\leq \zeta_k(A)$ if and only if $I\leq \zeta_{n+k}(A)$.
\end{lemma}
\begin{proof}
	Let us proceed by induction on $n$. If $n=0$ the statement is trivial. 
	Now we want to prove, for any $n\geq 0$, 
	that $\Gamma_n(I)\leq \zeta_{k+1}(A)$ is equivalent to $\Gamma_{n+1}(I)\leq \zeta_k(A)$.
	But this is clear from definitions 
	\eqref{centers} and \eqref{gammas}.
\end{proof}


%
%
%

\begin{theorem}
Let $(A,+,\circ)$ be a skew brace. The following are equivalent:
\begin{itemize}
\item[(i)] $(A,+,\circ )$ is centrally nilpotent of class $n$.
\item[(ii)] $\zeta_n(A)=A$.
\item[(iii)] $\Gamma_n(A)=0$.
\end{itemize}
\end{theorem}


\begin{proof}
(i) $\Leftrightarrow$ (ii) is due to Lemma~\ref{lem:lower_central_series}.\newline
(ii) $\Leftrightarrow$ (iii) According to Lemma \ref{gamma and zeta}, $\Gamma_n(A)=0=\zeta_0(A)$ if and only if $A\leq \zeta_n(A)$ 
\end{proof}

\begin{corollary}\label{nilpotent are both left and right}
Let $(A,+,\circ)$ be a centrally nilpotent skew brace. Then $(A,+,\circ)$ is a $*$-nilpotent brace of nilpotent type and $(A,\circ)$ is nilpotent.
\end{corollary}
\begin{proof}
It follows since $A^{n},A^{(n)}\leq \Gamma_n$ for every $n\in \mathbb{N}$ and the ideal $\zeta_1(A)$ is contained in $Z(A,+)\cap Z(A,\circ)$. 
\end{proof}

\begin{remark}
Let $(A,+)$ be a 
perfect group. The skew brace $(A,+,+)$ is $*$-nilpotent (indeed $x*y=0$ for every $x,y\in A$), but it is not centrally nilpotent since $A=[A,A]_+$.
\end{remark}


\begin{proposition}\label{center}
	Let $(A,+,\circ)$ be a $*$-nilpotent skew brace of nilpotent type and $I$ be a non trivial ideal of $A$. Then $\zeta(A)\cap I\neq 0$. In particular, $\zeta(A)\neq 0$.	
\end{proposition}

\begin{proof}
 Assume that $(A,+,\circ)$ is both left and right $*$-nilpotent. Then $\Soc{A}\cap I\neq 0$ since it is right $*$-nilpotent of nilpotent type (see \cite[Theorem 2.8]{NilpotentType}) and $I\cap \zeta(A)=I	\cap \Soc{A}\cap \Fix{A}\neq 0$ since it is left $*$-nilpotent (\cite[Proposition 2.26]{NilpotentType}). If $I=A$ we have that $\zeta(A)\neq 0$.
\end{proof}

This gives us that finite skew braces are nilpotent if and only if they are $*$-left and $*$-right nilpotent skew braces of nilpotent type.

\begin{corollary}\label{finite case}
Let $(A,+,\circ)$ be a finite skew brace. The following are equivalent
\begin{itemize}
\item[(i)] $(A,+,\circ)$ is a $*$-nilpotent skew brace of nilpotent type.
\item[(ii)] $(A,+,\circ)$ is right $*$-nilpotent skew brace of nilpotent type and $(A,\circ)$ is nilpotent.
\item[(iii)] $(A,+,\circ)$ is a centrally nilpotent skew brace.\end{itemize}
\end{corollary}
%
%

\begin{proof}
The equivalence	(i) $\Leftrightarrow$ (ii) follows by \cite[Theorem 4.8]{NilpotentType} and \cite[Theorem 2.20]{NilpotentType} and the implication (iii) $\Rightarrow$ (i) by Corollary \ref{nilpotent are both left and right}.\newline
(i) $\Rightarrow$ (iii) According to Proposition \ref{center} the series of the ideals in \eqref{centers} is ascending. Since $A$ is finite, then $A=\zeta_n(A)$ for some $n\in \mathbb{N}$.
\end{proof}

For skew braces of abelian type we can extend Corollary \ref{finite case} to the infinite case. Notice that
if $A$ is of abelian type then $\zeta(A)=\ker{\lambda}\cap Fix(A)$.

\begin{proposition}\label{nilpotent of abelian type}
	Let $(A,+,\circ)$ be a skew brace of abelian type. The following are equivalent:
	\begin{itemize}
		\item[(i)] $(A,+,\circ)$ is $*$-nilpotent. 
		\item[(ii)] $(A,+,\circ)$ is right $*$-nilpotent and $(A,\circ)$ is nilpotent.
		\item[(iii)] $(A,+,\circ)$ is centrally nilpotent.
\end{itemize}
\end{proposition}

\begin{proof}
	
The equivalence	(i) $\Leftrightarrow$ (ii) is \cite[Theorem 3]{Engel} and (iii) $\Rightarrow$ (i) is true in general (see Corollary \ref{nilpotent are both left and right}).\\
	(i) $\Rightarrow$ (iii) Assume that $(A,+,\circ)$ is both right and left $*$-nilpotent. Then there exists $n\in \mathbb{N}$ such that $A^{[n]}=0$ according to \cite[Theorem 3]{Engel}. Let proceed by induction on $n$.\\
	If $A^{[1]}=0$ then $A$ is the trivial group and in particular it is a nilpotent brace.
Assume that $A^{[n]}=0$, then $A^{[n-1]}*A=A*A^{[n-1]}=0$, therefore $A^{[n-1]}\leq \Fix{A}\cap \ker{\lambda}=\zeta(A)$.
Hence $\left(A/\zeta(A)\right)^{[n-1]}=0$. By induction, $A/\zeta(A)$ is centrally nilpotent and $A$ is centrally nilpotent as well.
\end{proof}

We are going to show a construction of braces, similar to the construction given in \cite[Section 9]{BCJO} using bilinear mappings. Actually,
this construction is a special case of the central extension given in~Theorem~\ref{cen_ext}.
Recall that bilinear mappings are special cases of group cocycles.
Let $H$ be a group and $K$ be an abelian group. We denote by $K\times_\theta H$ the central extension of $H$ by $K$ via the group cocycle $\theta$ (if $\theta=0$ then $K\times_\theta H$ is the usual direct product of groups).  

\begin{proposition}\label{with bilinear}
	Let $(H,+,\circ)$ be a centrally nilpotent skew brace of length $n$, $K$ be an abelian group and $\theta:H\times H\longrightarrow (K,+)$ be a bilinear map with respect to both $+$ and $\circ$. Then $A=((K,+)\times (H,+),(K,+)\times_\theta (H,\circ))$ is a centrally nilpotent brace of length less or equal to $n+1$.
\end{proposition}

\begin{proof}
	Let $h_1,h_2\in H$ and $k_1,k_2\in K$. Then
	\begin{align*}
	\lambda_{(k_1,h_1)}((k_2,h_2))&=(-k_1,-h_1)+(k_1,h_1)\circ (k_2,h_2)=\\
	&=(-k_1,-h_1)+(k_1+k_2+\theta(h_1,h_2),h_1\circ h_2)\\
	&=(k_2+\theta(h_1,h_2),-h_1+h_1\circ h_2)\\
	&=(k_2+\theta(h_1,h_2),\lambda_{h_1}(h_2)).
	\end{align*}
	Hence $\lambda_{(h_1,k_1)}$ is bijective and it is an automorphism of $(A,+)$ since $\theta$ is bilinear. Moreover $I=K\times \{0\}\leq \ker{\lambda}\cap Fix(\lambda)\cap Z(G,+)=\zeta(A)$ and in particular it is an ideal.
	Moreover
	\begin{align*}
	\lambda_{(k_1,h_1)+\lambda_{(k_1,h_1)}(k_2,h_2)}(x,y)&=\lambda_{(k_1+k_2+\theta(h_1,h_2),h_1+\lambda_{h_1}(h_2))}(x,y)\\
	&=\lambda_{(k_1+k_2+\theta(h_1,h_2),h_1\circ h_2)}(x,y)\\
	&=(x+\theta(h_1\circ h_2,y),\lambda_{h_1}\lambda_{h_2}(y))\\
		\lambda_{(k_1,h_1)}\lambda_{(k_2,h_2)}(x,y)&=\lambda_{(k_1,h_1)}((x+\theta(h_2,y),\lambda_{h_2}(y))\\
		&=((x+\theta(h_2,y)+\theta(h_1,\lambda_{h_2}(y)),\lambda_{h_1}\lambda_{h_2}(y))\\
		&=((x+\theta(h_2,y)+\theta(h_1,y),\lambda_{h_1}\lambda_{h_2}(y)).
	\end{align*}
	The two expression above are equal, so according to Lemma \ref{condition on lambda}, $A$ is a skew brace. Since $I\leq \zeta(A)$ and $A/I=(H,+,\circ)$ then $A$ is centrally nilpotent of length at most $n+1$.
\end{proof}

\begin{corollary}\label{with bilinear2}
	Let $H, K$ be abelian groups, $\theta:H\times H\longrightarrow K$ be a bilinear map. Then $A=(H\times K, H\times_\theta K)$ is a centrally nilpotent skew brace of length $2$.
\end{corollary}
%

\begin{example}
	Let $E,F,A$ be abelian groups and $\omega:E\times F\longrightarrow A$ be a bilinear mapping. Let $\mathbb{H}(\omega)=(E\times F\times A, \circ)$ is the generalized Heisenberg group defined as 
	\[(e_1,f_1,a_1)\circ (e_2,f_2,a_2)=
	 (e_1+e_2,f_1+f_2,a_1+a_2+\omega(e_1,f_2)),
	\]
	for every $e_1,e_2\in E$, $f_1,f_2\in F$ and $a_1,a_2\in A$. Actually, this operation is usually displayed in a matrix-like style, namely
	\begin{displaymath}
	\begin{bmatrix}
	1 & e_1 & a_1\\
	0 & 1 &f_1\\
	0 & 0 &1
	\end{bmatrix}\circ \begin{bmatrix}
	1 & e_2 & a_2\\
	0 & 1 &f_2\\
	0 & 0 &1
	\end{bmatrix}=\begin{bmatrix}
	1 & e_1+e_2 & a_1+a_2+\omega(e_1,f_2)\\
	0 & 1 &f_1+f_2\\
	0 & 0 &1
	\end{bmatrix}.
	\end{displaymath}
	 Then $(E\times F\times E, +,\circ)$ where $(E\times F\times A,+)$ is the direct product of the three groups and $(E\times F\times A, \circ)=\mathbb{H}(\omega)$ is a nilpotent brace of length $2$. See \cite{Heisenberg} for further details about Heisenberg groups.
\end{example}

\section{A universal algebraic motivation}

The motivation why we came up with the definition of the
central nilpotency is
 the notion of nilpotency coming from the universal algebraic commutator theory developed in \cite{comm}.
  The general theory is very technical
and therefore we present the notions in
the narrower class of expanded groups, a class that contains groups, rings, vector spaces, skew braces etc.

\begin{definition}
 An {\em expanded group} is a set~$A$ with several operations of several arities,
 among them a binary operation~$+$, a~unary operation~$-$ and a constant~$0$ such that
 $(A,+,-,0)$ is a group.\newline
 A $n$-ary {\em polynomial} of~$A$ is a function $A^n\to A$ such that there exist constants~$c_1,\ldots,c_k$ and a $n+k$-ary term~$t$ such that $f(x_1,\ldots,x_n)=t(c_1,\ldots,c_k,x_1,\ldots,x_n)$.
 We denote by $\Pol_n(A)$ the set of all
 $n$-ary polynomials of~$A$.
\end{definition}

 Examples of expanded groups are, for instance,
 rings or vector spaces.
 In the class of rings, polynomials are the classical polynomial functions. In the class of vector spaces, polynomials are linear combinations plus a vector.
 Polynomials are used to define commutators.
 
 \begin{definition} Let~$A$ be an expanded group.
  Let $f(x_1,\ldots,x_n)\in\Pol_n(A)$.
  We say that $f$ is {\em absorbing} if,
  for all $1\leq i\leq n$ and all $a_j\in A$,
  with $1\leq j\leq n$,
  $f(a_1,\ldots,a_{i-1},0,a_{i+1},\ldots,a_n)=0$.
  A~subset~$I$ of an expanded group~$A$ is called an~\emph{ideal} if it there exists an endomorphism~$\varphi$ of~$A$ such that $\varphi(i)=0$ if and only if~$i\in I$.
  Let $I,J$ be two ideals of~$A$. A {\em commutator} $\comm IJ$ of two ideals is defined as the ideal generated by the set
  \[ \{ f(i,j) \ : \ i\in I, j\in J, f\in\Pol_2(A),\> f \text{ is absorbing}\}.\]
 \end{definition}

 In the class of groups, the only binary polynomials with $f(x,0)=f(0,x)=0$ are some combinations of (element-wise) commutators and therefore $\comm IJ=[I,J]$. In the class of rings, the commutator
 is the smallest ideal containing $I\cdot J+J\cdot I$, since $f=x\cdot y$ is a binary absorbing polynomial.
 In the class of vector spaces, there are no non-constant absorbing polynomials and
 therefore the commutator of two subspaces is the trivial subspace.

\begin{proposition}\label{necessary cond for [,]=1}
Let $(A,+,\circ)$ be a skew brace and $I,J$ be ideals. Then $I*J+J*I+[I,J]_+\leq \comm{I}{J}$.
\end{proposition}
\begin{proof}
Consider the absorbing polynomials $x*y$, $y*x$ and $x+y-x-y$.
\end{proof}

\begin{problem}
Let $(A,+,\circ)$ be a skew brace and $I,J$ be ideals. Is it true that $\comm{I}{J}$ is the smallest ideal containing $I*J+J*I+[I,J]_+$? Does the equality $ \comm{I}{J}=I*J+J*I+[I,J]_+$ hold?

\end{problem}


\begin{definition}\label{def of abelian}
 Let $A$ be an expanded group. We say, that an ideal~$I$ is {\em central}
 if $\comm AI=0$. 
 We say that $A$ is {\em abelian} if $A$ is central in~$A$.
 The {\em center} of~$A$
 is the largest central ideal of~$A$.
\end{definition}

\begin{example}We give a few examples of centers of expanded groups:

\begin{itemize}

\item[(i)] For a ring~$R$, an ideal is central if and only if it lies in~$\mathrm{Ann}_R(R)$. Therefore, the center of a ring is its annihilator. And a ring is abelian if and only if $R\cdot R=0$.
\item[(ii)] Every vector space is abelian since the commutator of two subspaces is always trivial.
\item[(iii)] A prototypical binary absorbing polynomial for Lie algebras is the Lie bracket. Therefore
the center of a Lie algebra is the radical of the Lie bracket.

\end{itemize}
\end{example}

\begin{corollary}\label{cor:central_ideals}
Let $(A,+,\circ)$ be a skew brace. Then every central ideal of~$A$ is contained in $\zeta(A)$. 
\end{corollary}

In the theory of left braces the word {\em abelian} is already used for left braces with both the groups abelian, not necessarily isomorphic. Which is actually a much wider class than abelian skew braces in the sense of the commutator theory.

\begin{corollary}\label{Cor:Abelianness}
Let $(A,+,\circ)$ be a skew brace. The following are equivalent:
\begin{itemize}
\item[(i)] $(A,+,\circ)$ is abelian (in the sense of Definition \ref{def of abelian}).
\item[(ii)] $(A,+)=(A,\circ)$ is an abelian group.
\item[(iii)] $A*A= [A,A]_+=0$.
\end{itemize}
\end{corollary}

Fortunately, the notion of a {\em center} of a skew brace has not been used until the beginning of our Section~2. Now we shall prove that this definition is conforming with the commutator theory,
that means, that $\zeta(A)$ is the center of~$A$ in the sense of Definition~\ref{def of abelian}.
For our proofs
concerning the centralizing relation for ideals, we work with terms of skew braces as the terms in the language $\{+,-,*,\overline{(\, )},0 \}$:
it is not hard to see that this language is equivalent since, having a term
in the language $\{+,-,\circ,\overline{(\, )},0 \}$, we can replace every occurrence of $\circ$
using
\[x\circ y=x+(x*y)+y.\]
Actually, this is the path along which years ago two-sided braces evolved from radical rings.


%
%



\begin{lemma}\label{comm free terms}
Let $(A,\cdot,\circ)$ be a skew brace and let $t(x_1,\ldots, x_n)$ a term in which the variable $x_1$ is not involved in any $*$ operation. Then there exists $k\in \mathbb{Z}$ such that 
$t(x_1+z,\ldots, x_n)=t(x_1,\ldots, x_n)+kz$ for every $z\in \zeta(A)$ and every $x_1,\ldots, x_n\in A$.
\end{lemma}


\begin{proof}
We may assume that the variable $x_1$ is actually not present in~$t$. In this case, the statement is vacuously true.

Suppose that $x_1$ is present in~$t$ and we prove the lemma by induction on the height of~$t$. 
 Recall that if $x\in \zeta(A)$ then $z\in Z(A,+)\cap Z(A,\circ)$ and $-z=\overline{z}$. If the height is zero, the statement is true.
Suppose hence that the height is positive. The root operation is not~$*$ hence there are three possibilities:\newline
1) $t=-s$; then 
\[t(x_1+z,\ldots,x_n)=-s(x_1+z,\ldots,x_n)=-(s(x_1,\ldots,x_n)+kz)=
t(x_1,\ldots,x_n)-kz.\]
2) $t=\bar s$; then 
\begin{multline*}
t(x_1+k,\ldots,x_n)=\overline{s(x_1+z,\ldots,x_n)}=\overline{(s(x_1,\ldots,x_n)+kz)}=\overline{(s(x_1,\ldots,x_n)\circ kz)}\\
=\overline{kz}\circ t(x_1,\ldots,x_n)=
t(x_1,\ldots,x_n)-kz.
\end{multline*}
3) $t(x_1,\ldots, x_n)=s(x_{1},\ldots , x_{n})+r(x_{1},\ldots , x_{n})$; then
\begin{multline*}
t(x_1+z,\ldots, x_n )=
 s(x_1+z,\ldots,x_n)+r(x_1+z,\ldots,x_n)=\\
 s(x_1,\ldots,x_n)+nz+r(x_1,\ldots,x_n)+mz.
 =t(x_1,\ldots,x_n)+(n+m)x.
 \end{multline*}
\end{proof}

\begin{theorem}
Let $(A,+,\circ)$ be a skew brace. Then $\zeta(A)$ is the biggest central ideal of $A$.
\end{theorem}

\begin{proof}
According to Corollary~\ref{cor:central_ideals}, we only need to prove $\comm{A}{\zeta(A)}=0$, that means, 
$f(x,y)=0$,
for any binary absorbing polynomial~$f$, $x\in A$ and $y\in\zeta(A)$.

Let $f(y,x)=t(x,x_1,\ldots,x_n)$ and let $t(0,y_1,\ldots,y_n)=f(y,0)=0$. Consider one of the lowest node in which the operation $*$ appears and involves the variable $x$. Then you have a subterm which look like this $s(y_1,\ldots,y_m)*r(x,x_1,\ldots,x_n)$, where $s$ and $r$ are terms (or $r(x,x_1,\ldots,x_n)*s(y_1,\ldots,y_m)$). The subterm $r(x,x_1,\ldots,x_n)$ contains just group operations $+$ and inversions involving $x$, then using Lemma \ref{comm free terms} there exists $k\in \mathbb{Z}$ such that $r(x,x_1,\ldots,x_n)=r(0,x_1,\ldots,x_n)+kx$ for every $x\in \zeta(A)$. Therefore, assuming that $x\in \zeta(A)$ we have 
\begin{eqnarray*}
s(y_1,\ldots,y_m)*r(x,x_1,\ldots,x_n) &=&
s(y_1,\ldots,y_m)*(r(0,x_1,\ldots,x_n)+kx)=\\
&\overset{\eqref{comm1}}{=}&s(y_1,\ldots,y_m)*r(0,x_1,\ldots,x_n).
\end{eqnarray*}
Hence, if $x\in \zeta(A)$, we can remove $x$ from such a node and we can apply the same argument to the next node $*$ involving the variable $x$. So we can run through every branch of the term removing $x$ from every $*$ node. 

Thus, we can assume that $x$ is not involved in any $*$ operation. Then, by Lemma~\ref{comm free terms}, we have 
\begin{align*}
t(z,y,\ldots,y_n)&=t(0,y_1,\ldots,y_n)+kz=f(y,0)+kz=0+kz\\
&=t(0,0,\ldots,0)+kz=t(z,0,\ldots,0)=f(0,z)=0.
\end{align*}
Thus we get that $f(y,x)=t(x,y_1,\ldots,y_n)=t(0,y_1,\ldots,y_n)=f(y,0)=0$.
Therefore $\zeta(A)$ is a central ideal. 
\end{proof}

Skew braces form a Maltsev variety and for such structures there exists a standard construction named \emph{the central extension} 
given in Proposition 7.1 of \cite{comm}. 
We shall present the construction in the context of braces.
Recall, that for braces $\zeta(B)=Z(B,\cdot)\cap \mathrm{Soc}(B)$.
\begin{theorem}\label{cen_ext}
Let $B$ be a skew brace. Then $B\cong (B/\zeta(B)\times \zeta(B),+,\circ)$, where
\begin{eqnarray}
(x,a)+(y,b)&=&(x+y,a+b+\theta(x,y))\label{plus}\\
(x,a)\cdot (y,b)&=&(x\cdot y,a+b+\phi(x,y))\label{dot}\\
\phi(x,y+z)-\phi(x,y)-\phi(x,z) &=&\theta((x\circ y)-x,x\circ z)+\theta(x\circ y,-x)-\theta(x,-x)-\theta(y,z)\label{cocycle},
%
\end{eqnarray}
for every $x,y,z\in B/\zeta(B)$, $a,b,c\in \zeta(B)$, where $\theta:(B/\zeta(B),+)^2\to \zeta(B)$ is an abelian group cocycle and $\phi:(B/\zeta(B),\circ )^2\to \zeta(B)$ is a group cocycle.
\end{theorem}
\begin{proof}
Clearly $(B,+)$ is a central extension of $(B/\zeta(B),+)$, so the operation $+$ is described as in \eqref{plus} by a group cocycle $\phi$.  Since $\zeta(B)\leq Z(B,\circ)$, then and $(B/\zeta(B),\circ)$ is a central extension of $(B/\zeta(B),\circ)$, so \eqref{dot} holds since the operations $+$ and $\circ$ coincides on $\zeta(B)$. The last condition \eqref{cocycle}, follows directly by the compatibility condition between $+$ and $\circ$. 
\end{proof}


At the end we shall mention that
there exists another 
universal algebraic
definition that generalizes the notion of nilpotent groups.

\begin{definition}
 Let~$A$ be an expanded group.
 We define {\em higher commutators} as follows: let $I_1,\ldots,I_k$ be ideals
 of~$A$ then 
 $\llbracket I_1,\ldots,I_k\rrbracket$
 is the ideal generated by
 \[\{f(x_1,\ldots,x_k)\ : \ x_i\in I_i, f\in\Pol_k(A),\> f\text{ absorbing}\}.\]
 We say that $A$
 is {\em supernilpotent of class $k$}
 if $\llbracket \underbrace{A,A,\ldots,A}_{k+1\text{ times}}\rrbracket=0$.
\end{definition}

Another reformulation of the definition
is that $A$ is $k$-supernilpotent if
and only if every $k+1$-ary absorbing polynomial is constant.
It is known \cite{AE} that for a group, the notions of nilpotency and of supernilpotency coincide. In general, the supernilpotency of expanded groups is a stronger property.

\begin{theorem}\cite[Corollary 6.15]{AM}
 Let $A$ be a skew brace which is supernilpotent of class~$k$. Then $A$ is centrally nilpotent of class at most~$k$.
\end{theorem}

\begin{problem}
A converse implication does not hold for expanded groups in general. However, some converse may be true for skew braces.
\end{problem}

\bibliographystyle{plain}
\bibliography{references} 

\end{document}